\documentclass[12pt]{amsart}
\usepackage{amssymb,amsmath,amsthm,mathrsfs,multirow,xcolor,framed,url}
\usepackage[pdftex,
         pdfauthor={Dandan Chen, Rong Chen and Mengjie Zhao},
         pdftitle={Combinatorial proof},
         pdfsubject={MATHEMATICS},
         pdfkeywords={Total number of ones, Crank for partitions, Theta-functions.},
         pdfproducer={Latex with hyperref},
         pdfcreator={pdflatex}]{hyperref}
\usepackage [latin1]{inputenc}
\newtheorem{theorem}{Theorem}[section]
\newtheorem{lemma}[theorem]{Lemma}
\newtheorem{cor}[theorem]{Corollary}

\theoremstyle{definition}

\theoremstyle{remark}

\numberwithin{equation}{section}

\newcommand\nutwid{\overset {\text{\lower 3pt\hbox{$\sim$}}}\nu}

\allowdisplaybreaks

\newcommand\omycite[1]{}

\newcommand{\beqs}{\begin{equation*}}
\newcommand{\eeqs}{\end{equation*}}
\newcommand{\beq}{\begin{equation}}
\newcommand{\eeq}{\end{equation}}

\begin{document}
\title[Combinatorial proofs]{Combinatorial proofs for partitions with repeated smallest part}


\author{Dandan Chen}
\address{Department of Mathematics, Shanghai University, People's Republic of China}
\address{Newtouch Center for Mathematics of Shanghai University, Shanghai, People's Republic of China}
\email{mathcdd@shu.edu.cn}
\author{Rong Chen}
\address{Department of Mathematics, Shanghai Normal University, Shanghai, China}
\email{rchen@shnu.edu.cn}
\author{Mengjie Zhao}
\address{Department of Mathematics, Shanghai University, People's Republic of China}
\email{zmjj@shu.edu.cn}


\subjclass[2010]{05A17, 11P87}

\date{}


\keywords{q-series; Integer partitions }

\begin{abstract}

Recently, Andrews and El Bachraoui considered the number of integer partitions whose smallest part is repeated exactly $k$ times and the remaining parts are not repeated. They presented several interesting results and  posed questions regarding combinatorial proofs for these identities. In this paper, we establish bijections to provide  combinatorial proofs for these results.

\end{abstract}

\maketitle


\section{Introduction}


 Recently, Andrews and El Bachraoui \cite[Defination 1]{AB-JMAA-25} defined Spt$k_d(n)$ as the set of partitions $\pi$ of $n$ where the smallest part $s(\pi)$ occurs exactly $k$ times and the remaining parts do not repeat and spt$k_d(n)$ as the cardinality of Spt$k_d(n)$.
They denoted $A_{0}(k,n)(resp.,A_{1}(k,n)$ as the number of partitions $\pi$ in the set Spt$k_d(n)$ where the number of parts that are greater than $s(\pi)$ is even (resp., odd) and denoted spt$k'_d(n)$=$A_{0}(k,n)-A_{1}(k,n)$.
Then clearly,
\begin{align*}
\sum_{n\geq1}sptk_d(n)q^n=\sum_{n\geq1}q^{kn}(-q^{n+1};q)_\infty~\text{and}~
\sum_{n\geq1}sptk^{'}_d(n)q^n=\sum_{n\geq1}q^{kn}(q^{n+1};q)_\infty.
\end{align*}
Here and throughout this paper, we adopt the notations \cite{Ga+Ra-04}
 $$(a;q)_0=1, (a;q)_n=\prod_{j=0}^{n-1}(1-aq^j), (a;q)_{\infty}=\prod_{j=0}^{\infty}(1-aq^j)~ for ~|q|<1. $$

\begin{theorem}\cite[Theorem 1]{AB-JMAA-25}\label{thm1}
  For any positive integer $k$ we have
 $$\sum_{n\ge 1}sptk_{d}(n)q^n=P_{k}(q)(-q;q)_{\infty}+(-1)^k(q;q)_{k-1},$$

where
$$P_{k}(q)=\begin{cases}
               1  & \mbox{for }k=1, \\
               (q^{k-1}-1)P_{k-1}(q)+q^{k-1}  & \mbox{for }k>1.
             \end{cases} $$

\end{theorem}



Andrews and El Bachraoui \cite[Defination 2]{AB-JMAA-25} also defined Spt$k_{do}(n)$ as the set of partitions $\pi$ of $n$ where the smallest part $s(\pi)$ occurs exactly $k$ times, and the remaining parts are distinct and incongruent modulo 2 with $s(\pi)$.  The cardinality of this set is denoted by spt$k_{do}(n)$. They further defined $B_{0}(k,n)(resp.,B_{1}(k,n)$ as the number of partitions $\pi$ in the set Spt$k_{do}(n)$ where the number of parts greater than $s(\pi)$ is even (resp., odd), and they denoted spt$k'_{do}(n)$=$B_{0}(k,n)-B_{1}(k,n)$.
Then clearly,
\begin{align*}
\sum_{n\geq1}sptk_{do}(n)q^n=\sum_{n\geq1}q^{kn}(-q^{n+1};q^2)_\infty~\text{and}~
\sum_{n\geq1}sptk^{'}_{do}(n)q^n=\sum_{n\geq1}q^{kn}(q^{n+1};q^2)_\infty.
\end{align*}
\begin{theorem}\cite[Theorem 3]{AB-JMAA-25}\label{thm2}
  For any positive integer k we have

 $$ \sum_{n\ge 1}sptk'_{do}(n)q^n=T_{k}(q)(q;q^2)_{\infty}+2q^k(q^2;q^2)_{k-1},$$
where
$$T_{k}(q)=\begin{cases}
               -q  & \mbox{for }k=1, \\
               (q-q^{2k-1})T_{k-1}(q)-q^{2k-1}  & \mbox{for }k>1.
             \end{cases} $$
\end{theorem}

\begin{theorem}\cite[Theorem 5]{AB-JMAA-25}\label{thm3}
  For any positive integer k we have

 $$ \sum_{n\ge 1}sptk_{do}(n)q^n=V_{k}(q)(-q^2;q^2)_{\infty}+W_{k}(q)(-q;q^2)_{\infty}+2(-q)^k(q^2;q^2)_{k-1},$$
where
$$V_{k}(q)=\begin{cases}
               2q  & \mbox{for }k=1, \\
               (q^{2k-1}-q)V_{k-1}(q)+2q^k  & \mbox{for }k>1,
             \end{cases} $$
and
$$W_{k}(q)=\begin{cases}
               q  & \mbox{for }k=1, \\
              (q^{2k-1}-q)W_{k-1}(q)+q^{2k-1}  & \mbox{for }k>1.
             \end{cases} $$
\end{theorem}

We note that $sptk_{d}(n)=sptk_{do}(n)=sptk'_{do}(n)=B_0(k,n)=B_1(k,n)=0$ for all $n<k$.

Analytic proofs for the Theorems \ref{thm1}--\ref{thm3} were established by Andrews and El Bachraoui, who employed mathematical induction to derive the corresponding generating functions and also provided applications of these theorems to partitions, as seen in Corollaries \ref{r:cor1}--\ref{r:cor3} below. At the end of their paper \cite{AB-JMAA-25}, they requested combinatorial proofs for these corollaries.

In this paper, we offer alternative combinatorial proofs for Theorems \ref{thm1}--\ref{thm3} by rewriting them as Theorems \ref{thm1-1} and \ref{thm2-2}. The Corollaries \ref{r:cor1}--\ref{r:cor3} can be immediately proved using Theorems \ref{thm1-1} and \ref{thm2-2}.

\begin{cor}\cite[Cor. 5]{AB-JMAA-25}
\label{r:cor1}
(a) For any integer $n\geq 2$ we have
\begin{align}\label{eq-cor-1}
spt2_d(n)=2p_d(n-1)-p_d(n).
\end{align}
(b) For any integer $n\geq 4$ we have
\begin{align}\label{eq-cor-2}
spt3_d(n)=2p_d(n-3)-2p_d(n-1)+p_d(n).
\end{align}
\end{cor}

\begin{cor}\cite[Cor. 7]{AB-JMAA-25}
\label{r:cor2}
 For any integer $n\geq 2$ we have
\begin{align*}
spt1_{do}(n)=2p_{de}(n-1)+p_{do}(n-1).
\end{align*}
\end{cor}

\begin{cor}\cite[Cor. 8]{AB-JMAA-25}
\label{r:cor3}
(a) For any integer $n\geq 2$ we have
\begin{align*}
spt1'_{do}(n)=-p'_{do}(n-1).
\end{align*}
(b) For any integer $n\geq 5$ we have
\begin{align*}
spt2'_{do}(n)=p'_{do}(n-4)-p'_{do}(n-3)-p'_{do}(n-2).
\end{align*}
\end{cor}

This paper is organized as follows. In Section 2, we present a combination proof of Theorem \ref{thm1}. In Section 3, we introduce the  recurrence relations associated with $B_{0}(k,n)$ and $B_{1}(k,n)$ to provide combinatorial proofs for Theorems \ref{thm2} and \ref{thm3}.

\section{Combinatorial proof of Theorem \ref{thm1} }
Let $\pi$ denote a partition of positive integer $n$. The smallest part of $\pi$ is denoted by $s(\pi)$, and the second smallest part is denoted by $ss(\pi)$.
Additionally, let $p_d(n)$ be the number of partitions of $n$ into distinct parts for which it is well-known that \cite{An-76}
 \begin{align*}
 \sum_{n\geq 0}p_d(n)q^n=(-q;q)_\infty \quad \text{  (assuming that $p_d(0)=1$)}.
 \end{align*}



\begin{theorem}\label{thm1-1}
(a) For ${k}\ge 2$ we have
\begin{align*}
    sptk_{d}(n)+spt(k-1)_{d}(n)=spt(k-1)_{d}(n-k+1)+p_d(n-k+1).
\end{align*}
(b) For $n\geq 2$, we have
\begin{align*}
    spt1_{d}(n)=p_d(n).
\end{align*}

\end{theorem}
\begin{proof}
Case (b) is obvious. We prove case (a). Define $\phi$: Spt$k_{d}(n)$ $\cup$ Spt$(k-1)_{d}(n)$ $\rightarrow$ Spt$(k-1)_{d}(n-k+1)$ $\cup$ $\pi_{p_d}(n-k+1)$.
Let $\lambda$ $\in$ Spt$k_{d}(n)$ $\cup$ Spt$(k-1)_{d}(n)$ such that $\phi(\lambda)=\mu$, where $\mu$ is constructed as follows.

(1)If $\lambda\in$ Spt$(k-1)_{d}(n)$, we remove one $``1"$ from each of the $k-1$ repeated parts, and denote the resulting partition as $\mu$, which satisfies:
  \begin{itemize}
  \item If $s(\lambda)=1$, then
   $\mu\in$ $\pi_{p_d}(n-k+1)$ with $s(\mu)>1$;

  \item If $s(\lambda)>1$, then
   $\mu\in$ Spt$(k-1)_{d}(n-k+1)$ with $s(\mu)=s(\lambda)-1<ss(\mu)-1$.
\end{itemize}

(2)If $\lambda\in$ Spt$k_{d}(n)$, we remove one $"1"$ from each of the first $k-1$ repeated parts, and denote the resulting partition as $\mu$, which satisfies:
  \begin{itemize}
  \item If $s(\lambda)=1$, then
   $\mu\in$ $\pi_{p_d}(n-k+1)$ with $s(\mu)=1$;

  \item If $s(\lambda)>1$, then
   $\mu\in$ Spt$(k-1)_{d}(n-k+1)$ with  $s(\mu)=s(\lambda)-1=ss(\mu)-1$.
\end{itemize}
Thus, $\mu\in$ Spt$(k-1)_{d}(n-k+1)$ $\cup$ $\pi_{p_d}(n-k+1)$.

Conversely, we define $\psi$: Spt$(k-1)_{d}(n-k+1)$ $\cup$ $\pi_{p_d}(n-k+1)$ $\rightarrow$ Spt$k_{d}(n)$ $\cup$ Spt$(k-1)_{d}(n)$. Let $\mu$ $\in$ Spt$(k-1)_{d}(n-k+1)$ $\cup$ $\pi_{p_d}(n-k+1)$ such that $\psi(\mu)=\lambda$, where $\lambda$ is constructed as follows.

(1)If $\mu\in$ Spt$(k-1)_{d}(n-k+1)$, we add one $1$ to each of the $k-1$ repeated parts in $\mu$, and denote the resulting partition as $\lambda$, which satisfies:
\begin{itemize}
  \item If $ss(\mu)=s(\mu)+1$, then
   $\lambda\in$ Spt$k_{d}(n)$ with $s(\lambda)=s(\mu)+1>1$;

  \item If $ss(\mu)>s(\mu)+1$, then
   $\lambda\in$ Spt$(k-1)_{d}(n)$ with $s(\lambda)=s(\mu)+1>1$.
\end{itemize}

(2)If $\mu\in$  $\pi_{p_d}(n-k+1)$, we add $k-1$ ones as parts to $\mu$, and denote the resulting partition  as $\lambda$, which satisfies:

\begin{itemize}
  \item If $s(\mu)=1$, then
   $\lambda\in$ Spt$k_{d}(n)$ with $s(\lambda)=1$;

  \item If $s(\mu)>1$, then
   $\lambda\in$ Spt$(k-1)_{d}(n)$ with $s(\lambda)=1$.

\end{itemize}
Thus, $\lambda\in$ Spt$k_{d}(n) \cup$ Spt$(k-1)_{d}(n)$.

Then we have $\psi\circ \phi(\lambda)=\lambda$ and $\phi\circ \psi(\mu)=\mu$, i.e. $\psi\circ \phi=\phi\circ \psi=id$. It follows that $\phi$: Spt$k_{d}(n)$ $\cup$ Spt$(k-1)_{d}(n)$ $\rightarrow$ Spt$k_{d}(n-k)$ $\cup$ $\pi_{p_d(n-k)}$ is a bijection, implying that  spt$k_{d}(n)+$spt$(k-1)_{d}(n)=$spt$(k-1)_{d}(n-k+1)+p_d(n-k+1)$.
\end{proof}

We can now ready to prove Theorem \ref{thm1} by mathematical induction on $k$.
Since $spt1_d(n)=p_d(n)$ for $n\geq 1$,  we have
 \begin{align*}
 \sum_{n\ge1}spt1(n)q^n=\sum_{n\ge1}p_d(n)q^n=(-q;q)_\infty-1,
 \end{align*}
  which is the desired identity for $k=1$. Now assuming that the result holds for $k-1$, by Theorem \ref{thm1-1}, we obtain
%

\begin{align*}
  &\sum_{n\ge 1}sptk_{d}(n)q^n=\sum_{n\ge k}sptk_{d}(n)q^n\\
  =&\sum_{n\ge k}spt(k-1)_{d}(n-k+1)q^n-\sum_{n\ge k}spt(k-1)_{d}(n)q^n+\sum_{n\ge k}p_d(n-k+1)q^n\\
  =&\sum_{n\ge 1}spt(k-1)_{d}(n)q^{n+k-1}-\sum_{n\ge k}spt(k-1)_{d}(n)q^n+\sum_{n\ge 1}p_d(n)q^{n+k-1}\\
    =&q^{k-1}\sum_{n\ge 1}spt(k-1)_{d}(n)q^n-\bigg(\sum_{n\ge 1}spt(k-1)_{d}(n)q^n\\
    &-\sum_{n=1}^{k-1} spt(k-1)_{d}(n)q^n\bigg)+q^{k-1}(\sum_{n\ge 0}p_d(n)q^n-p_d(0)q^0)\\
    =&(q^{k-1}-1)(P_{k-1}(q)(-q;q)_\infty+(-1)^{k-1}(q;q)_{k-2})+q^{k-1}(-q;q)_\infty\\
    =&P_{k}(q)(-q;q)_\infty+(-1)^{k}(q;q)_{k-1}.
\end{align*}

This completes the proof of Theorem \ref{thm1}.
\section{Combinatorial proofs of Theorems \ref{thm2} and \ref{thm3} }

 Let $t(\pi)$  be the number of parts in partition $\pi$ that are greater than $s(\pi)$, and let $\ell(\pi)$ be the length of $\pi$. Recall that $\pi_{B0}(k,n)$ (resp., $\pi_{B1}(k,n)$) represents  the set of partitions in the set Spt$k_{do}(n)$ an even (resp., odd) number of parts  greater than $s(\pi)$. Also, we recall that $\pi_{p_d}(n)$ denotes the set of partitions of $n$ with all  distinct parts.
  Denote $\pi_{pde}(n)$(resp., $\pi_{pdo}(n))$ as the set of partitions of $n$ with all  distinct even (resp., odd) parts. The number of such partitions is  $p_{de}(n)$ (resp., $p_{do}(n)$). Let $p'_{do}(n)$ be the difference between the number of partitions counted by $p_{do}(n)$ with an even  number of parts and those with an odd number of parts.  Note that \cite{An-09}
  \begin{align}
  \label{r:q1}
  (-q^2;q^2)_\infty=\sum_{n\geq 0}p_{de}(n)q^n  \text{\qquad (assuming that $p_{de}(0)=1$)},\\
  \label{r:q2}
  (-q;q^2)_\infty=\sum_{n\geq 0}p_{do}(n)q^n \text{\qquad (assuming that $p_{do}(0)=1$)},\\
  \label{r:q3}
  (q;q^2)_\infty=\sum_{n\geq 0}p'_{do}(n)q^n  \text{\qquad (assuming that $p'_{do}(0)=1$)}.
  \end{align}


To establish Theorem \ref{thm2-2}, we need to prove Lemmas \ref{lem1}--\ref{lem4} first.

\begin{lemma}\label{lem1}
  For positive integers $k$, we have
  \begin{align*}
   B_0(k+1,2t+1)+B_1(k,2t)=B_1(k,2t-2k)+p_{de}(2t-k).
  \end{align*}
\end{lemma}
\begin{proof}
When $k$ is odd,  the sets $\pi_{B0}(k+1,2t+1)$, $\pi_{B1}(k,2t)$, $\pi_{B1}(k,2t-2k)$ and $\pi_{pde}(2t-k)$  are all empty. Therefore, we only need consider the case when $k$ is even.

We define $\phi_1$: $\pi_{B0}(k+1,2t+1)$$\cup$$\pi_{B1}(k,2t)$ $\rightarrow$ $\pi_{B1}(k,2t-2k)$$\cup$$\pi_{pde}(2t-k)$.
For $\lambda$ $\in$ $\pi_{B0}(k+1,2t+1)$$\cup$$\pi_{B1}(k,2t)$, we construct $\phi_1(\lambda)=\mu$ as follows.


(1) If $\lambda\in$ $\pi_{B0}(k+1,2t+1)$, note that $t(\lambda)$ is even and $s(\lambda)$ is odd. We denote the resulting partition as $\mu$, which satisfies:
  \begin{itemize}
  \item If $s(\lambda)=1$,  we remove one $``1"$ from each of the $k+1$ repeated parts in $\lambda$. Then
   $\mu\in$ $\pi_{pde}(2t-k)$ with even $\ell(\mu)$;

  \item If $s(\lambda)>1$,  we remove $``2"$ from each of the first $k$ repeated parts and remove one $``1"$ from the last repeated part. Then
      $\mu\in$ $\pi_{B1}(k,2t-2k)$ with $s(\mu)=s(\lambda)-2$ and $ss(\mu)=s(\lambda)-1$.
\end{itemize}

(2) If $\lambda\in$ $\pi_{B1}(k,2t)$,  note that $t(\lambda)$ and $s(\lambda)$  are both odd. We denote the resulting partition as $\mu$, which satisfies:
  \begin{itemize}
  \item If $s(\lambda)=1$, we remove one $``1"$ from each of the $k$ repeated parts in $\lambda$. Then
   $\mu\in$ $\pi_{pde}(2t-k)$ with odd $\ell(\mu)$;

  \item If $s(\lambda)>1$, we remove one $``2"$ from each of the $k$ repeated parts. Then
      $\mu\in$ $\pi_{B1}(k,2t-2k)$ with $s(\mu)=s(\lambda)-2$ and  $ss(\mu)>s(\lambda)$.
\end{itemize}
Thus, $\mu\in$ $\pi_{pde}(2t-k)$ $\cup$ $\pi_{B1}(k,2t-2k)$.

Conversely, we define $\psi_1$: $\pi_{B1}(k,2t-2k)$$\cup$$\pi_{pde}(2t-k)$ $\rightarrow$  $\pi_{B0}(k+1,2t+1)$$\cup$$\pi_{B1}(k,2t)$. For $\mu$ $\in$ $\pi_{pde}(2t-k)$ $\cup$ $\pi_{B1}(k,2t-2k)$, we construct $\psi_1(\mu)=\lambda$ as follows.

(1) If $\mu\in$ $\pi_{B1}(k,2t-2k)$, note that the remaining parts are distinct and incongruent modulo $2$ with $s(\mu)$. We denote the resulting partition as $\lambda$, which satisfies:
  \begin{itemize}
  \item If $ss(\mu)=s(\mu)+1$, we add one $``2"$ to each of the $k$ repeated parts in $\mu$, and add one $``1"$ to the second smallest part. Then
   $\lambda\in$ $\pi_{B0}(k+1,2t+1)$ with $s(\lambda)=s(\mu)+2\geq 3$;

  \item If $ss(\mu)>s(\mu)+1$, we add one $``2"$ to each of the  $k$ repeated parts. Then
      $\lambda\in$ $\pi_{B1}(k,2t)$ with $s(\lambda)=s(\mu)+2\geq 3$.
\end{itemize}

(2) If $\mu\in$ $\pi_{pde}(2t-k)$, note that $ s(\mu)$ is even. We denote the resulting partition as $\lambda$, which satisfies:
  \begin{itemize}
  \item If $\ell(\mu)$ is odd, we add $k$ ones as parts in $\mu$. Then
   $\lambda\in$ $\pi_{B1}(k,2t)$  with  $s(\lambda)=1$;

  \item If $\ell(\mu)$ is even, we add $k+1$ ones as parts in $\mu$. Then
   $\lambda\in$ $\pi_{B0}(k+1,2t+1)$  with  $s(\lambda)=1$.
\end{itemize}
Thus, $\lambda\in$ $\pi_{B0}(k+1,2t+1)$$\cup$$\pi_{B1}(k,2t)$.

Then we have $\psi_1\circ \phi_1(\lambda)=\lambda$ and $\phi_1\circ \psi_1(\mu)=\mu$, i.e. $\psi_1\circ \phi_1=\phi_1\circ \psi_1=id$. It follows that $\phi_1$: $\pi_{B0}(k+1,2t+1)$$\cup$$\pi_{B1}(k,2t)$ $\rightarrow$ $\pi_{B1}(k,2t-2k)$$\cup$$\pi_{pde}(2t-k)$ is a bijection, i.e.  $B_0(k+1,2t+1)+B_1(k,2t)=B_1(k,2t-2k)+p_{de}(2t-k)$.

\end{proof}
\begin{lemma}\label{lem2}
  For  positive integers $k$, we have
  \begin{align*}
    B_0(k+1,2t)+B_1(k,2t-1)=B_1(k,2t-2k-1)+p_{de}(2t-k-1)+p_{do}(2t-2k-1).
  \end{align*}
\end{lemma}
\begin{proof}

We define $\phi_2$: $\pi_{B0}(k+1,2t)$$\cup$$\pi_{B1}(k,2t-1)$ $\rightarrow$ $\pi_{B1}(k,2t-2k-1)$$\cup$$\pi_{pde}(2t-k-1)$$\cup$$\pi_{pdo}(2t-2k-1)$.
For $\lambda$ $\in$ $\pi_{B0}(k+1,2t)$$\cup$$\pi_{B1}(k,2t-1)$, we construct  $\phi_2(\lambda)=\mu$ as follows.

(1) If $\lambda\in$ $\pi_{B0}(k+1,2t)$, note that $t(\lambda)$ is even. We denote the resulting partition as $\mu$, which satisfies:

    $\quad$(a) If $s(\lambda)=1$, we remove one $``1"$ from each of the $k+1$ repeated parts in $\lambda$. Then
   $\mu\in$ $\pi_{pde}(2t-k-1)$ with even $\ell(\mu)$;

      $\quad$(b) If $s(\lambda)\ge2$, we remove $``2"$ from each of the first $k$ repeated parts and remove one $``1"$ from the last repeated part. Then
 \begin{itemize}
  \item If $s(\lambda)=2$, then $\mu\in$ $\pi_{pdo}(2t-2k-1)$ with $s(\mu)=s(\lambda)-1$;
  \item If $s(\lambda)\ge3$, then $\mu\in$ $\pi_{B1}(k,2t-2k-1)$ with $s(\mu)=s(\lambda)-2$ and $ss(\mu)=s(\lambda)-1$.

\end{itemize}

(2) If $\lambda\in$ $\pi_{B1}(k,2t-1)$, note that $t(\lambda)$ is odd.  We denote the resulting partition as $\mu$, which satisfies:

      $\quad$(a) If $s(\lambda)=1$, we remove one $``1"$ from each of the $k$ repeated parts in $\lambda$. Then
   $\mu\in$ $\pi_{pde}(2t-k-1)$ with odd $\ell(\mu)$;

      $\quad$(b) If $s(\lambda)\ge2$, we remove $``2"$ from each of the $k$ repeated parts. Then
 \begin{itemize}
  \item If $s(\lambda)=2$, then $\mu\in$ $\pi_{pdo}(2t-2k-1)$ with $s(\mu)>s(\lambda)$;
  \item If $s(\lambda)\ge3$, then $\mu\in$ $\pi_{B1}(k,2t-2k-1)$ with $s(\mu)=s(\lambda)-2$ and $ss(\mu)>s(\lambda)$.

\end{itemize}
Thus, $\mu\in$ $\pi_{B1}(k,2t-2k-1)$$\cup$$\pi_{pde}(2t-k-1)$$\cup$$\pi_{pdo}(2t-2k-1)$.

Conversely, we define $\psi_2$: $\pi_{B1}(k,2t-2k-1)$$\cup$$\pi_{pde}(2t-k-1)$$\cup$$\pi_{pdo}(2t-2k-1)$ $\rightarrow$ $\pi_{B0}(k+1,2t)$$\cup$$\pi_{B1}(k,2t-1)$.
 For $\mu$ $\in$ $\pi_{B1}(k,2t-2k-1)$$\cup$$\pi_{pde}(2t-k-1)$$\cup$$\pi_{pdo}(2t-2k-1)$, we construct $\psi_2(\mu)=\lambda$ as follows.

(1) If $\mu\in$ $\pi_{B1}(k,2t-2k-1)$, note that $t(\mu)$ is odd.  We denote the resulting partition as $\lambda$, which satisfies:

$\quad$  a) If $ss(\mu)=s(\mu)+1$, we add one $``2"$ to each of the $k$ repeated parts in $\mu$, and add one $``1"$ to the second smallest part in $\mu$. Then 
  \begin{itemize}
  \item If $s(\mu)$ is odd,
   $\lambda\in$ $\pi_{B0}(k+1,2t)$ with $s(\lambda)=s(\mu)+2 \geq3$ and non repeated even parts;

  \item If $s(\mu)$ is even,
   $\lambda\in$ $\pi_{B0}(k+1,2t)$ with $s(\lambda)=s(\mu)+2 \geq3$ and non repeated odd parts;
\end{itemize}

$\quad$  b) If $ss(\mu)>s(\mu)+1$, we add one $``2"$ to each of the $k$ repeated parts in $\mu$. Then 
  \begin{itemize}
  \item If $s(\mu)$ is odd,
   $\lambda\in$ $\pi_{B1}(k,2t-1)$ with $s(\lambda)=s(\mu)+2, ss(\lambda)>s(\mu)+1$ and non repeated even parts;

  \item If $s(\mu)$ is even,
   $\lambda\in$ $\pi_{B1}(k,2t-1)$ with $s(\lambda)=s(\mu)+2, ss(\lambda)>s(\mu)+1$ and non repeated odd parts;
\end{itemize}

(2) If $\mu\in$ $\pi_{pde}(2t-k-1)$, note that $s(\mu)$ is even. We denote the resulting partition as $\lambda$, which satisfies:
  \begin{itemize}
  \item If $\ell(\mu)$ is even, we add $k+1$ ones as parts in $\mu$. Then
   $\lambda\in$ $\pi_{B0}(k+1,2t)$  with  $s(\lambda)=1$ and non repeated even parts;

  \item If $\ell(\mu)$ is odd, we add $k$ ones as parts in $\mu$. Then
   $\lambda\in$ $\pi_{B1}(k,2t-1)$  with  $s(\lambda)=1$ and non repeated even parts;
\end{itemize}

(3) If $\mu\in$ $\pi_{pdo}(2t-2k-1)$, note that $l(\mu)$ and $s(\mu)$  are both odd. We denote the resulting partition as $\lambda$, which satisfies:
  \begin{itemize}
  \item If $s(\mu)=1$, we add a $``1"$ to the smallest part and add $k$ twos as parts in $\mu$. Then
   $\lambda\in$ $\pi_{B0}(k+1,2t)$  with  $s(\lambda)=2$ and non repeated odd parts;

  \item If $s(\mu)>1$, we add $k$ twos as parts in $\mu$. Then
   $\lambda\in$ $\pi_{B1}(k,2t-1)$  with  $s(\lambda)=2$ and non repeated odd parts.
\end{itemize}
Thus, $\lambda\in$ $\pi_{B0}(k+1,2t)$$\cup$$\pi_{B1}(k,2t-1)$.

Then we have $\psi_2\circ \phi_2(\lambda)=\lambda$ and $\phi_2\circ \psi_2(\mu)=\mu$, i.e. $\psi_2\circ \phi_2=\phi_2\circ \psi_2=id$. It follows that $\phi_2$: $\pi_{B0}(k+1,2t)$$\cup$$\pi_{B1}(k,2t-1)$ $\rightarrow$ $\pi_{B1}(k,2t-2k-1)$$\cup$$\pi_{pde}(2t-k-1)$$\cup$$\pi_{pdo}(2t-2k-1)$ is a bijection, i.e.  $B_1(k,2t-1)+B_0(k+1,2t)=B_1(k,2t-2k-1)+p_{de}(2t-k-1)+p_{do}(2t-2k-1)$.

\end{proof}

\begin{lemma}\label{lem3}
  For positive integers $k$, we have
  \begin{align*}
   B_0(k,2t-1)+B_1(k+1,2t)=B_0(k,2t-2k-1)+p_{de}(2t-k-1).
  \end{align*}
\end{lemma}
\begin{proof}
When $k$ is even, the sets $\pi_{B0}(k,2t-1)$, $\pi_{B1}(k+1,2t)$, $\pi_{B0}(k,2t-2k-1)$ and $\pi_{pde}(2t-k-1)$  are all empty. Therefore, we only need consider the case when $k$ is odd.

We define $\phi_3$: $\pi_{B0}(k,2t-1)$$\cup$$\pi_{B1}(k+1,2t)$ $\rightarrow$ $\pi_{B0}(k,2t-2k-1)$$\cup$$\pi_{pde}(2t-k-1)$.
For $\lambda$ $\in$ $\pi_{B0}(k,2t-1)$$\cup$$\pi_{B1}(k+1,2t)$, we construct  $\phi(\lambda)=\mu$ as follows.

(1) If $\lambda\in$ $\pi_{B0}(k,2t-1)$, note that $t(\lambda)$ is even and $s(\lambda)$ is odd.  We denote the resulting partition as $\mu$, which satisfies:
  \begin{itemize}
  \item If $s(\lambda)=1$, we remove one $``1"$ from each of the $k$ repeated parts in $\lambda$. Then
   $\mu\in$ $\pi_{pde}(2t-k-1)$ with even $\ell(\mu)$;

  \item If $s(\lambda)>1$, we remove $``2"$ from each of the $k$ repeated parts. Then
      $\mu\in$ $\pi_{B0}(k,2t-2k-1)$ with $s(\mu)=s(\lambda)-2$ and $ss(\mu)>s(\lambda)$.
\end{itemize}

(2) If $\lambda\in$ $\pi_{B1}(k+1,2t)$, note that $t(\lambda)$  and $s(\lambda)$  are both odd.  We denote the resulting partition as $\mu$, which satisfies:
  \begin{itemize}
  \item If $s(\lambda)=1$, we remove one $``1"$ from each of the $k+1$ repeated parts in $\lambda$. Then
   $\mu\in$ $\pi_{pde}(2t-k-1)$ with odd $\ell(\mu)$;

  \item If $s(\lambda)>1$, we remove one $``2"$ from each of the first $k$ repeated parts and remove a one from the last repeated part. Then
      $\mu\in$ $\pi_{B0}(k,2t-2k-1)$ with $s(\mu)=s(\lambda)-2$ and $ss(\mu)=s(\lambda)-1$.
\end{itemize}
Thus, $\mu\in$ $\pi_{B0}(k,2t-2k-1)$$\cup$$\pi_{pde}(2t-k-1)$.

Conversely, we define $\psi_3$:  $\pi_{B0}(k,2t-2k-1)$ $\cup$ $\pi_{pde}(2t-k-1)$ $\rightarrow$ $\pi_{B0}(k,2t-1)$$\cup$$\pi_{B1}(k+1,2t)$. For $\mu$ $\in$ $\pi_{B0}(k,2t-2k-1)$$\cup$$\pi_{pde}(2t-k-1)$, we construct $\psi_3(\mu)=\lambda$ as follows.

(1) If $\mu\in$ $\pi_{B0}(k,2t-2k-1)$, note that $t(\mu)$ is even. We denote the resulting partition as $\lambda$, which satisfies:
  \begin{itemize}
  \item If $ss(\mu)=s(\mu)+1$, we add one $``2"$ to each of the $k$ repeated parts in $\mu$, and add one $``1"$ to the second smallest part. Then
   $\lambda\in$ $\pi_{B1}(k+1,2t)$ with $s(\lambda)=s(\mu)+2 \geq3$;

  \item If $ss(\mu)>s(\mu)+1$, we add one $``2"$ to each of the  $k$ repeated parts. Then
      $\lambda\in$ $\pi_{B0}(k,2t-1)$ with $s(\lambda)=s(\mu)+2 \geq3$.
\end{itemize}

(2) If $\mu\in$ $\pi_{pde}(2t-k-1)$, note that $s(\mu)$ is even.  We denote the resulting partition as $\lambda$, which satisfies:
  \begin{itemize}
  \item If $\ell(\mu)$ is even, we add $k$ ones as parts in $\mu$. Then
   $\lambda\in$ $\pi_{B0}(k,2t-1)$  with  $s(\lambda)=1$;

  \item If $\ell(\mu)$ is odd, we add $k+1$ ones as parts in $\mu$. Then
   $\lambda\in$ $\pi_{B1}(k+1,2t)$  with  $s(\lambda)=1$.
\end{itemize}
Thus, $\lambda\in$ $\pi_{B0}(k,2t-1)$$\cup$$\pi_{B1}(k+1,2t)$.

Then we have $\psi_3\circ \phi_3(\lambda)=\lambda$ and $\phi_3\circ \psi_3(\mu)=\mu$, i.e. $\psi_3\circ \phi_3=\phi_3\circ \psi_3=id$. It follows that $\phi_3$: $\pi_{B0}(k,2t-1)$$\cup$$\pi_{B1}(k+1,2t)$ $\rightarrow$ $\pi_{B0}(k,2t-2k-1)$$\cup$$\pi_{pde}(2t-k-1)$ is a bijection, i.e.  $B_0(k,2t-1)+B_1(k+1,2t)=B_0(k,2t-2k-1)+p_{de}(2t-k-1)$.

\end{proof}
\begin{lemma}\label{lem4}
  For positive integers $k$, we have
  \begin{align*}
    B_0(k,2t)+B_1(k+1,2t+1)=B_0(k,2t-2k)+p_{de}(2t-k)+p_{do}(2t-2k).
  \end{align*}
\end{lemma}
\begin{proof}
We define $\phi_4$: $\pi_{B0}(k,2t)$$\cup$$\pi_{B1}(k+1,2t+1)$ $\rightarrow$ $\pi_{B0}(k,2t-2k)$$\cup$$\pi_{pde}(2t-k)$$\cup$$\pi_{pdo}(2t-2k)$.
For $\lambda$ $\in$ $\pi_{B0}(k,2t)$$\cup$$\pi_{B1}(k+1,2t+1)$, we construct $\phi(\lambda)=\mu$  as follows.

(1) If $\lambda\in$ $\pi_{B0}(k,2t)$, note that $t(\lambda)$ is even.  We denote the resulting partition as $\mu$, which satisfies:

    $\quad$(a) If $s(\lambda)=1$, we remove one $``1"$ from each of the $k$ repeated parts in $\lambda$. Then
   $\mu\in$ $\pi_{pde}(2t-k)$ with even $\ell(\mu)$;

      $\quad$(b) If $s(\lambda)\ge2$, we remove $``2"$ from each of the $k$ repeated parts. Then
 \begin{itemize}
  \item If $s(\lambda)=2$, then $\mu\in$ $\pi_{pdo}(2t-2k)$ with $s(\mu)>2$;
  \item If $s(\lambda)\ge3$, then $\mu\in$ $\pi_{B0}(k,2t-2k)$ with $s(\mu)=s(\lambda)-2$ and $ss(\mu)> s(\lambda)$.
\end{itemize}

(2) If $\lambda\in$ $\pi_{B1}(k+1,2t+1)$, note that $t(\lambda)$ is odd. We denote the resulting partition as $\mu$, which satisfies:

      $\quad$ a) If $s(\lambda)=1$, we remove one $``1"$ from each of the $k+1$ repeated parts in $\lambda$. Then
   $\mu\in$ $\pi_{pde}(2t-k)$ with odd $\ell(\mu)$;

      $\quad$ b) If $s(\lambda)\ge2$, we remove $``2"$ from each of the first $k$ repeated parts and a $``1"$ from the last repeated part. Then
 \begin{itemize}
  \item If $s(\lambda)=2$, then $\mu\in$ $\pi_{pdo}(2t-2k)$ with $s(\mu)=1$;
  \item If $s(\lambda)\ge3$, then $\mu\in$ $\pi_{B0}(k,2t-2k)$ with $s(\mu)=s(\lambda)-2$ and $ss(\mu)=s(\lambda)-1$.
\end{itemize}
Thus, $\mu\in$ $\pi_{B0}(k,2t-2k)$$\cup$$\pi_{pde}(2t-k)$$\cup$$\pi_{pdo}(2t-2k)$.

Conversely, we define $\psi_4$:  $\pi_{B0}(k,2t-2k)$$\cup$$\pi_{pde}(2t-k)$$\cup$$\pi_{pdo}(2t-2k)$ $\rightarrow$ $\pi_{B0}(k,2t)$$\cup$$\pi_{B1}(k+1,2t+1)$. For $\mu$ $\in$ $\pi_{B0}(k,2t-2k)$$\cup$$\pi_{pde}(2t-k)$$\cup$$\pi_{pdo}(2t-2k)$,  we construct $\psi_4(\mu)=\lambda$ as follows.

(1) If $\mu\in$ $\pi_{B0}(k,2t-2k)$, note that $t(\mu)$ is even.  We denote the resulting partition as $\lambda$, which satisfies:

$\quad$  a) If $ss(\mu)=s(\mu)+1$, we add one $``2"$ to each of the $k$ repeated parts in $\mu$, and add one $"1"$ to the second smallest part. Then 
  \begin{itemize}
  \item If $s(\mu)$ is odd,
   $\lambda\in$ $\pi_{B1}(k+1,2t+1)$ with $s(\lambda)=s(\mu)+2$ and $ss(\lambda)>s(\mu)+1$ and non repeated even parts;

  \item if $s(\mu)$ is even,
   $\lambda\in$ $\pi_{B1}(k+1,2t+1)$ with $s(\lambda)=s(\mu)+2$ and $ss(\lambda)>s(\mu)+1$ and non repeated odd parts;
\end{itemize}

$\quad$  b) If $ss(\mu)>s(\mu)+1$, we add one $``2"$ to each of the $k$ repeated parts in $\mu$.
Then  
  \begin{itemize}
  \item If $s(\mu)$ is odd,
   $\lambda\in$ $\pi_{B0}(k,2t)$ with $s(\lambda)=s(\mu)+2$ and $ss(\lambda)>s(\mu)+1$ and non repeated even parts;

  \item If $s(\mu)$ is even,
   $\lambda\in$ $\pi_{B0}(k,2t)$ with $s(\lambda)=s(\mu)+2$ and  $ss(\lambda)>s(\mu)+1$ and non repeated odd parts;
\end{itemize}

(2) If $\mu\in$ $\pi_{pde}(2t-k)$, note that $s(\mu)$ is even. We denote the resulting partition as $\lambda$, which satisfies:
  \begin{itemize}
  \item If $\ell(\mu)$ is even, we add $k$ ones as parts in $\mu$. Then
   $\lambda\in$ $\pi_{B0}(k,2t)$  with  $s(\lambda)=1$ and non repeated even parts;

  \item If $\ell(\mu)$ is odd, we add $k+1$ ones as parts in $\mu$. Then
   $\lambda\in$ $\pi_{B1}(k+1,2t+1)$  with  $s(\lambda)=1$ and non repeated even parts.
\end{itemize}

(3) If $\mu\in$ $\pi_{pdo}(2t-2k)$, note that $\ell(\mu)$ is even and $s(\mu)$ is odd. We denote the resulting partition as $\lambda$, which satisfies:
  \begin{itemize}
  \item If $s(\mu)=1$, we add a $``1"$ to the smallest part and add $k$ twos as parts in $\mu$. Then
   $\lambda\in$ $\pi_{B1}(k+1,2t+1)$  with  $s(\lambda)=2$ and non repeated odd parts;

  \item If $s(\mu)>1$, we add $k$ twos as parts in $\mu$. Then
   $\lambda\in$ $\pi_{B0}(k,2t)$  with  $s(\lambda)=2$ and non repeated odd parts.
\end{itemize}
Thus, $\lambda\in$ $\pi_{B0}(k,2t)$$\cup$$\pi_{B1}(k+1,2t+1)$.

Then we have $\psi_4\circ \phi_4(\lambda)=\lambda$ and $\phi_4\circ \psi_4(\mu)=\mu$, i.e. $\psi_4\circ \phi_4=\phi_4\circ \psi_4=id$. It follows that $\phi_4$: $\pi_{B0}(k,2t)$$\cup$$\pi_{B1}(k+1,2t+1)$ $\rightarrow$ $\pi_{B0}(k,2t-2k)$$\cup$$\pi_{pde}(2t-k)$$\cup$$\pi_{pdo}(2t-2k)$ is a bijection, i.e.  $B_0(k,2t)+B_1(k+1,2t+1)=B_0(k,2t-2k)+p_{de}(2t-k)+p_{do}(2t-2k)$.
\end{proof}

With the same maps in Lemmas \ref{lem1}--\ref{lem4}, and noting that the sets $\pi_{B1}(1,2t)$ and $\pi_{pde}(2t-1)$ are both empty, we obtain
\begin{lemma}\label{lem5}
For positive integer $t$, we have
\begin{align*}
&B_0(1,2t+1)=p_{de}(2t),\\
&B_0(1,2t)=p_{de}(2t-1)+p_{do}(2t-1),\\
&B_1(1,2t)=p_{de}(2t-1),\\
&B_1(1,2t+1)=p_{de}(2t)+p_{do}(2t).
\end{align*}
\end{lemma}

\begin{theorem}\label{thm2-2}

(a) For $k\geq 2$, we have
\begin{align*}
&spt(k-1)'_{do}(n-1)-sptk'_{do}(n)=spt(k-1)'_{do}(n-2k+1)+p'_{do}(n-2k+1);\\
&sptk_{do}(n)+spt(k-1)_{do}(n-1)=spt(k-1)_{do}(n-2k+1)+2p_{de}(n-k)+p_{do}(n-2k+1).
\end{align*}
(b) For $n\geq 2$, we have
\begin{align*}
&spt1_{do}(n)=2p_{de}(n-1)+p_{do}(n-1),\\
&spt1'_{do}(n)=-p'_{do}(n-1).
\end{align*}
\end{theorem}

\begin{proof}
First, let $n=2t+1$ in Lemmas \ref{lem1} and \ref{lem4}, respectively.  For $k\geq 1$, we have
 \begin{align*}
  &B_0(k+1,n)+B_1(k,n-1)=B_1(k,n-2k-1)+p_{de}(n-k-1);\\
  &B_0(k,n-1)+B_1(k+1,n)=B_0(k,n-2k-1)+p_{de}(n-k-1)+p_{do}(n-2k-1).
  \end{align*}
Next, let $n=2t$ in Lemmas \ref{lem2} and \ref{lem3}, respectively.  For $k\geq 1$, we have
\begin{align*}
  &B_0(k+1,n)+B_1(k,n-1)=B_1(k,n-2k-1)+p_{de}(n-k-1)+p_{do}(n-2k-1);\\
&B_0(k,n-1)+B_1(k,n-1)=B_0(k,n-2k-1)+p_{de}(n-k-1).
\end{align*}
According to the parity of $n$, we derive the following
\begin{align*}
  2B_0(k+1,n)+2B_1(k,n-1)&=2B_1(k,n-2k-1)+2p_{de}(n-k-1)\\
  &+(1+(-1)^{n-2k})p_{do}(n-2k-1);\\
  2B_0(k,n-1)+2B_1(k+1,n)&=2B_0(k,n-2k-1)+2p_{de}(n-k-1)\\
  &+(1+(-1)^{n-2k-1})p_{do}(n-2k-1).
\end{align*}
%
%
%
Combining  the above two equations, we have, for integers $k\geq2$,

\begin{align}
\label{eq-thm2-1}
&B_0(k-1,n-2k+1)-B_1(k-1,n-2k+1)+(-1)^{n-2k+1}p_{do}(n-2k+1)\\
\nonumber
=&B_0(k-1,n-1)-B_1(k-1,n-1)-B_0(k,n)+B_1(k,n),
\end{align}
and
\begin{align}
\label{eq-thm2-2}
&B_0(k,n)+B_1(k,n)+B_0(k-1,n-1)+B_1(k-1,n-1)\\
\nonumber
=&B_0(k-1,n-2k+1)+B_1(k-1,n-2k+1)+2p_{de}(n-k)+p_{do}(n-2k+1).
\end{align}
By Lemma \ref{lem5}, we can prove, for $n\geq 2$,
\begin{align*}
2B_0(1,n)=2p_{de}(n-1)+\left(1+(-1)^n\right)p_{do}(n-1),\\
2B_1(1,n)=2p_{de}(n-1)+\left(1-(-1)^n\right)p_{do}(n-1).
\end{align*}
Recall the definitions
 \begin{align*}
 sptk'_{do}(n)=B_{0}(k,n)-B_{1}(k,n)~~~~~ \text{and} ~~~~~~sptk_{do}(n)=B_{0}(k,n)+B_{1}(k,n).
\end{align*}
This completes the proof of Theorem \ref{thm2-2}.
\end{proof}

Next, we show that Theorem \ref{thm2-2} provides a combinatorial formulation to Theorems \ref{thm2}--\ref{thm3}.
\begin{proof}[Proof of Theorem \ref{thm2}]  By Theorem \ref{thm2-2},  we have
\begin{align*}
\sum_{n\ge1}spt1'_{do}(n)q^n=&q+\sum_{n\ge2}spt1'_{do}(n)q^n
=q-\sum_{n\ge2}p'_{do}(n-1)q^n\\
=&q-\sum_{n\ge1}p'_{do}(n)q^{n+1}=2q-q(q;q^2)_\infty,
\end{align*}
which is the desired identity for $k =1$. Now assume that the result holds for $k-1$. By Theorem \ref{thm2-2} and \eqref{r:q3}, we obtain
\begin{align*}
  &\sum_{n\ge1}sptk'_{do}(n)q^n=\sum_{n= 1}^{2k-2}sptk'_{do}(n)q^n+sptk'_{do}(2k-1)q^{2k-1}+\sum_{n\ge 2k}sptk'_{do}(n)q^n\\
  =&\sum_{n=1}^{2k-2}sptk'_{do}(n)q^n+spt(k-1)'_{do}(2k-2)q^{2k-1}-q^{2k-1}+\sum_{n\ge 2k}sptk'_{do}(n)q^n\\
 =&\left(\sum_{n=1}^{2k-2}sptk'_{do}(n)q^n+spt(k-1)'_{do}(2k-2)q^{2k-1}+\sum_{n\ge 2k}spt(k-1)'_{do}(n-1)q^n\right)\\
 &-\sum_{n\ge 2k}spt(k-1)'_{do}(n-2k+1)q^n
 -\sum_{n\ge 2k}p'_{do}(n-2k+1)q^n-q^{2k-1}\\
 =&\sum_{n\ge 1}spt(k-1)'_{do}(n)q^{n+1}-\sum_{n\ge 1}spt(k-1)'_{do}(n)q^{n+2k-1}-\sum_{n\ge0}p'_{do}(n)q^{n+2k-1}\\
 =&q\left(T_{k-1}(q)(q;q^2)_{\infty}+2q^{k-1}(q^2;q^2)_{k-2}\right)\\
 &-q^{2k-1}\left(T_{k-1}(q)(q;q^2)_{\infty}+2q^{k-1}(q^2;q^2)_{k-2}\right)
 -q^{2k-1}(q;q^2)_\infty\\
 =&\left((q-q^{2k-1})T_{k-1}(q)-q^{2k-1}\right)(q;q^2)_\infty+2q^{k}(q^2;q^2)_{k-1}\\
 =&T_{k}(q)(q;q^2)_{\infty}+2q^{k}(q^2;q^2)_{k-1}.
 \end{align*}
 This completes the proof of Theorem \ref{thm2}.
 \end{proof}

\begin{proof}[Proof of Theorem \ref{thm3}]By Theorem \ref{thm2-2},  we have
\begin{align*}
\sum_{n\ge1}spt1_{do}(n)q^n=&q+\sum_{n\ge2}spt1_{do}(n)q^n
=q+\sum_{n\geq2}(2p_{de}(n-1)+p_{do}(n-1))q^n\\
=&q+\sum_{n\geq1}(2p_{de}(n)+p_{do}(n))q^{n+1}
=-2q+2q(-q^2;q^2)_\infty+q(-q;q^2)_\infty,
\end{align*}
which is the desired identity for $k=1$. Now, assume that the result holds for $k$. By Theorem \ref{thm2-2} and \eqref{r:q1}--\eqref{r:q2}, we obtain
\begin{align*}
  &\sum_{n\ge1}sptk_{do}(n)q^n=\sum_{n=k}^{2k-2}sptk_{do}(n)q^n+sptk_{do}(2k-1)q^{2k-1}+\sum_{n\ge 2k}sptk_{do}(n)q^n\\
  &=\sum_{n=k}^{2k-2}\left(2p_{de}(n-k)-spt(k-1)_{do}(n-1)\right)q^n
  +\sum_{n\ge 2k}sptk_{do}(n)q^n\\
  &\qquad+\left(2p_{de}(k-1)+1-spt(k-1)_{do}(2k-2)\right)q^{2k-1}\\
  &=\sum_{n=k}^{2k-1}\left(2p_{de}(n-k)-spt(k-1)_{do}(n-1)\right)q^n
  +\sum_{n\ge 2k}spt(k-1)_{do}(n-2k+1)q^n\\
  &-\sum_{n\ge 2k}spt(k-1)_{do}(n-1)q^n+2\sum_{n\ge 2k}p_{de}(n-k)q^n+\sum_{n\ge 2k}p_{do}(n-2k+1)q^n+q^{2k-1}\\
  &=\sum_{n\ge 2k}spt(k-1)_{do}(n-2k+1)q^n-\sum_{n\ge k}spt(k-1)_{do}(n-1)q^n\\
  &~~~\qquad~~~~~~~+2\sum_{n\ge k}p_{de}(n-k)q^n+\sum_{n\ge 2k}p_{do}(n-2k+1)q^n+q^{2k-1}\\
  &=\sum_{n\ge 1}spt(k-1)_{do}(n)q^{n+2k-1}-\sum_{n\ge 1}spt(k-1)_{do}(n)q^{n+1}+2q^{k}(-q^2;q^2)_\infty+q^{2k-1}(-q;q^2)_\infty\\
  &=(q^{2k-1}-q)\left(V_{k-1}(q)(-q^2;q^2)_{\infty}+W_{k-1}(q)(-q;q^2)_{\infty}+2(-q)^{k-1}(q^2;q^2)_{k-2}\right)\\
  &~~~\qquad~~~~~~~+2q^{k}(-q^2;q^2)_\infty+q^{2k-1}(-q;q^2)_\infty\\
  &=\left((q^{2k-1}-q)V_{k-1}(q)+2q^{k}\right)(-q^2;q^2)_\infty\\
  &~~~\qquad~~~~~~~+\left((q^{2k-1}-q)W_{k-1}(q)+q^{2k-1}\right)(-q;q^2)_\infty+2(-q)^{k}(q^2;q^2)_{k-1}\\
  &=V_{k}(q)(-q^2;q^2)_{\infty}+W_{k}(q)(-q;q^2)_{\infty}+2(-q)^{k}(q^2;q^2)_{k-1}.
\end{align*}
This completes the proof of Theorem \ref{thm3}.
\end{proof}

\subsection*{Acknowledgements}
The authors were  supported by the National Key R\&D Program of China (Grant No. 2024YFA1014500) and the National Natural Science Foundation of China (Grant No. 12201387 and No. 12401438).






\end{document}